\documentclass[10pt, a4paper]{article}
\usepackage{amssymb}
\usepackage{amsmath}
\usepackage{amsfonts}
\usepackage{amscd}
\usepackage{geometry}
\usepackage{mathdots}
\usepackage[all,cmtip]{xy}
\usepackage{makeidx}
\usepackage{titlesec}
\usepackage{fancyhdr}

\newcommand{\C}{\mathbb{C}}
\newcommand{\R}{\mathbb{R}}
\newcommand{\Q}{\mathbb{Q}}
\newcommand{\N}{\mathbb{N}}
\newcommand{\Z}{\mathbb{Z}}
\newcommand{\A}{\mathbb{A}}
\newcommand{\F}{\mathbb{F}}

\DeclareMathOperator{\mul}{\mathsf{m}}

\DeclareMathOperator{\proj}{proj}
\DeclareMathOperator{\HT}{HT}
\DeclareMathOperator{\WD}{WD}

\DeclareMathOperator{\rec}{rec}
\DeclareMathOperator{\Art}{Art}

\DeclareMathOperator{\Hom}{Hom}

\DeclareMathOperator{\Aut}{Aut}

\DeclareMathOperator{\Ind}{Ind}
\DeclareMathOperator{\Frob}{Frob}

\DeclareMathOperator{\Gal}{Gal}
\DeclareMathOperator{\GL}{GL}
\DeclareMathOperator{\GSp}{GSp}
\DeclareMathOperator{\PGL}{PGL}
\DeclareMathOperator{\SL}{SL}
\DeclareMathOperator{\SU}{SU}

\DeclareMathOperator{\PSL}{PSL}
\DeclareMathOperator{\PGSp}{PGSp}
\DeclareMathOperator{\PSp}{PSp}
\DeclareMathOperator{\Sp}{Sp}
\DeclareMathOperator{\PO}{PO}
\DeclareMathOperator{\SO}{SO}
\DeclareMathOperator{\PSO}{PSO}
\DeclareMathOperator{\PGO}{PGO}
\DeclareMathOperator{\GO}{GO}
\DeclareMathOperator{\POM}{P\Omega}

\DeclareMathOperator{\POMO}{P\Gamma O}
\DeclareMathOperator{\TInd}{\otimes -Ind}

\titleformat{\section}[hang]
{\normalfont\filright\large}{\thesection. }{0pt}
{\upshape\bfseries}

\titleformat{\subsection}[hang]
{\itshape}{\thesubsection \ - }{0pt}
{}

\usepackage{amsthm}
\theoremstyle{plain}
\newtheorem{theo}{Theorem}[section]

\newtheorem{lemm}[theo]{Lemma}
\newtheorem{coro}[theo]{Corollary}

\theoremstyle{remark}
\newtheorem{rema}[theo]{\sc Remark}
\theoremstyle{definition}
\newtheorem{defi}[theo]{Definition}

\title{On the images of the Galois representations attached to certain
RAESDC automorphic representations of $\GL_n(\A_\Q)$}
\author{\small ADRI$\acute{\mbox{A}}$N ZENTENO  \footnote{Facultad de Ciencias, Universidad Nacional Aut\'onoma de M\'exico. Circuito Exterior s/n, Coyoac\'an, Cd. Universitaria, 04510, Mexico. \texttt{matematicazg@ciencias.unam.mx}}}
\date{\today}
\begin{document}

\maketitle

\begin{abstract}
In the 80's Aschbacher classified the maximal subgroups of almost all of the finite almost simple classical groups. Essentially, this classification divide these subgroups into two types. The first of these consist roughly of subgroups that preserve some kind of geometric structure, so they are commonly called subgroups of geometric type. 
In this paper we will prove the existence of infinitely many compatible systems $\{ \rho_\ell \}_\ell$ of $n$-dimensional Galois representations associated to regular algebraic, essentially self-dual, cuspidal automorphic representations of $\GL_n(\A_\Q)$ ($n$ even) such that, for almost all primes $\ell$, the image of $\overline{\rho}_{\ell}$ (the semi-simplification of the reduction of $\rho_\ell$) cannot be contained in a maximal subgroup of geometric type of an $n$-dimensional symplectic or orthogonal group. 
Then, we apply this result to some 12-dimensional representations to give heuristic evidence towards the inverse Galois problem for even-dimensional orthogonal groups.

\end{abstract}

\section{Introduction}\label{sec:intro}

A strategy to address the inverse Galois problem over $\Q$ consists of studding the image of continuous representations of the absolute Galois group $G_\Q :=\Gal(\overline{\Q}/\Q)$. More precisely, let $\overline{\rho} : G_\Q \rightarrow \GL_n(\F_{\ell^s})$ be a continuous representation.  As $\ker \overline{\rho} \subseteq G_{\Q}$ is an open subgroup, there exists a finite Galois extension $K/\Q$ such that $\ker \overline{\rho} = G_K := \Gal(\overline{K}/K)$. Therefore 
\[
\mbox{Im} \overline{\rho} \simeq G_\Q / \ker \overline{\rho} \simeq G_\Q /G_K \simeq \Gal(K/\Q).
\] 
This reasoning shows that, whenever we are given a Galois representation of $G_\Q$ over a finite field $\F_{\ell^s}$, we obtain a realization of $\mbox{Im} \overline{\rho} \subseteq \GL_n(\F_{\ell^s})$ as a Galois group of $\Q$.

By using this strategy and the notion of good-dihedral representations (or some generalization), introduced by Khare and Wintenberger in their proof of Serre's modularity conjecture \cite{KW}, several cases of the inverse Galois problem for projective symplectic groups and odd-dimensional orthogonal groups have been proved. We refer  the reder to \cite{AdR15} and \cite{wie} for a discussion on previous results towards the inverse Galois problem that have been proved by using this method and its variants.

Following this idea, in this paper we show that there exist compatible systems $\mathcal{R} = \{\rho_\ell \}_\ell$ of $n$-dimensional Galois representations $\rho_\ell:G_\Q \rightarrow \GL_n(\overline{\Q}_\ell)$ ($n$ even) such that the image of $\overline{\rho}_\ell$ (the semi-simplification of the reduction of $\rho_\ell$) cannot be contained in a maximal subgroup of ``geometric type" of an $n$-dimensional symplectic or orthogonal group for almost all $\ell$.

More precisely, we start by introducing the notion of maximally induced representations of $S$-type and $O$-type (see Definiton \ref{minduced}), which are a generalization of the good-dihedral representations of Khare and Wintenberger \cite{KW}. 
With this tool and the classification of the maximal subgroups of the finite almost simple classical groups (due to Aschbacher \cite{As84}) we prove a representation-theoretic result which gives us a set of local conditions needed to construct compatible systems $\mathcal{R} = \{\rho_\ell \}_\ell$ of $n$-dimensional Galois representations such that, for almost all $\ell$, the image of $\overline{\rho}_\ell$ is not contained in a maximal subgroup of geometric type (see Definition \ref{geot}) of an $n$-dimensional symplectic or orthogonal group.

A well known fact in arithmetic geometry is that, given a RAESDC (regular algebraic, essentially self-dual, cuspidal) automorphic representation of $\GL_n(\A_\Q)$, there exist a compatible system associated to it. Then, in order to prove the existence of compatible systems as mentioned above, we prove the existence of RAESDC automorphic representations of $\GL_n(\A_\Q)$ with certain appropriate local types (see Section \ref{sec:7}). To do this we use Arthur's work \cite{Ar12} on endoscopic classification of automorphic representations for symplectic and orthogonal groups combined with some slightly modified results of Shin \cite{Shi12}.

Finally, by using the classification of the maximal almost simple groups contained in the symplectic and orthogonal groups of small dimension (at most $12$) we give a refinement of our main result. More precisely, we prove that if $\rho_\ell$ is a maximally induced representation of $S$-type and $6 \leq n \leq 12$, the image of $\overline{\rho}^{\proj}_\ell$ (the projectivization of $\overline{\rho}$,  i.e., $\overline{\rho}$ composed with the natural projection $\GL (\F_{\ell^s}) \twoheadrightarrow \PGL_n(\F_{\ell^s})$, where $\PGL_n(\F_{\ell^s}) := \GL_n(\F_{\ell^s}) /(\F_{\ell^s}^\times \cdot \mbox{Id})$) is either $\PSp_n(\F_{\ell^s})$ or $\PGSp_n(\F_{\ell^s})$ and if $\rho_\ell$ is a maximally induced representation of $O$-type and $n = 12$, the image of $\overline{\rho}^{\proj}_\ell$ is is either $\POM^+_{12}(\F_{\ell^s})$, $\PSO^+_{12}(\F_{\ell^s})$, $\PO^+_{12}(\F_{\ell^s})$ or $\PGO^+_{12}(\F_{\ell^s})$, for some integer $s>0$. Here, whenever $G$ is a subgroup of a linear group $\GL_n(\F_{\ell^s})$, $\mbox{P}G$ will denote the projectivization of $G$ which is defined as the image of $G$ in $\PGL_n(\F_{\ell^s})$. 

An immediate consequence of this results is that the previous groups occurs as a Galois group over $\Q$ for infinitely many primes $\ell$ and infinitely many positive integers $s$.
To the best of our knowledge, the orthogonal groups mentioned above are not previously known to be Galois over $\Q$, except for $s=1$ which was studied in \cite{Zy14}. The symplectic case was previously studied in \cite{AdRDW}, \cite{ADW14}, \cite{ADSW14} and \cite{KLS}.

\subsection*{Acknowledgments}
This paper is part of my PhD thesis, then I am very grateful to my advisor Luis V. Dieulefait for his constant guidance and encouragement. I also thank to Sara Arias de Reyna for many stimulating conversations about Galois representations and Sug Woo Shin for a useful correspondence about the construction of automorphic representations with prescribed local conditions. 
Part of this work has been written during a stay at the Mathematical Institute of the University of Barcelona. I would like to thank this institution for their support and the optimal working conditions. 
Finally, I want to give special thanks to the anonymous referee, whose comments and suggestions have greatly improved the presentation and readability of this paper.
My research was supported by the CONACYT grant no. 432521/286915.

\subsection*{Notation}
Here we list some notation to be used throughout the article. If $K$ is a perfect field, we denote by $\overline{K}$ an algebraic closure of $K$ and by $G_K$ the absolute Galois group $\Gal(\overline{K} / K)$.
Let $\chi_\ell$ denote the $\ell$-adic cyclotomic character and $\overline{\chi}_\ell$ its reduction modulo $\ell$.
If $K$ is a number field or a finite extension of $\Q_p$, we denote by $\mathcal{O}_K$ its ring of integers. For a maximal ideal $\mathfrak{p}$ of $\mathcal{O}_K$, we let $D_\mathfrak{p}$ and $I_\mathfrak{p}$ be the corresponding decomposition and inertia group at $\mathfrak{p}$ respectively, and we denote by $\Frob_\mathfrak{p}$ the geometric Frobenius. 
Finally, $\WD(\rho)^{F-ss}$ will denote the Frobenius semisimplification of the Weil-Deligne representation attached to a representation $\rho$ of $G_{\Q_p}$ and $\rec$ is the notation for the Local Langlands Correspondence which attaches to an irreducible admissible representation of $\GL_n(\Q_p)$ a Weil-Deligne representation of the Weil group $W_{\Q_p}$ as in \cite{HT01}. Finally if $G$ is a finite group we denote by $G^{(i)}$ the $i$-th derived subgroup of $G$ and we use $G^{\infty}$ to denote $\bigcap_{i\geq0}$ $G^{(i)}$.


\section{Preliminaries on classical groups}\label{sec:2}

It is well known that there is a lack of consistency in the literature for the notation used for the classical groups. Hence, we include this section in order to fix the notation for orthogonal and symplectic groups that we will use through this paper. Our main references are \cite{BHR13} and \cite{KL90}.

Let $n$ be a positive integer, $K$ a field of characteristic different from 2 and $V$ an $n$-dimensional $K$-vector space with a non-degenerate bilinear pairing $\langle \cdot, \cdot \rangle$. We define the \emph{similitude group} $\Delta(V)$ of $\langle \cdot , \cdot \rangle$ as 
\[
\{ g \in \GL(V)  :  \langle gv, gw \rangle = \mul(g) \langle v,w \rangle, \mbox{ with } \mul(g) \in K^*, \mbox{ for all }v,w \in V \}.
\]
The character $\mul : \Delta (V) \rightarrow K^*$ is called the \emph{multiplier} (or \emph{similitude factor}). The \emph{isometry group} of $\langle \cdot , \cdot \rangle$ is the subgroup $I(V)$ of $\Delta(V)$ of elements with multiplier 1 and the \emph{special group} of $\langle \cdot, \cdot \rangle$ is the subgroup $S(V)$ of $\Delta(V)$ consisting of all matrices with determinant $1$. 

Let $\mathcal{B} = \{ e_1, \ldots , e_n\}$ be a basis of $V$. We define the matrix of the pairing $\langle \cdot , \cdot \rangle$ with respect to $\mathcal{B}$ as $J = (b_{ij})_{n \times n}$, where $b_{ij} = \langle e_i , e_j \rangle$ for all $i$ and $j$.
In particular, if $\langle \cdot , \cdot \rangle $ is alternating, it can be shown that $n$ is even and that we can choose a basis such that the matrix of $\langle \cdot , \cdot \rangle$ has the standard form
\[
J := \left( 
\begin{matrix}
0 & S \\
-S & 0
\end{matrix}
\right)
\quad \mbox{with} \quad S := 
\left( 
\begin{smallmatrix}
0 &  & 1 \\
 & \iddots &  \\
1 &  & 0
\end{smallmatrix}
\right) \in \mbox{M}_{\frac{n}{2}}.
\]
Then we can define the \emph{symplectic similitude group} of the alternating pairing $\langle \cdot, \cdot \rangle$ as $\GSp_n(K):= \Delta(V)$ and the \emph{symplectic group} of $\langle \cdot , \cdot \rangle$ as $\Sp_n(K) := I(V)$. Note that in this case all elements of $\Sp_n(K)$ have determinant one, then $\Sp_n(K) = S(V)$ too.

On the other hand, if $\langle \cdot , \cdot \rangle$ is a symmetric pairing, we define the \emph{orthogonal similitude group} of $\langle \cdot , \cdot \rangle$ as $\GO(V) := \Delta(V)$ and the \emph{orthogonal group} of $\langle \cdot , \cdot \rangle$ as $\mbox{O}(V):= I(V)$, whose elements have determinant $\pm 1$. Finally, we define the \emph{special orthogonal group} of $\langle \cdot , \cdot \rangle$ as $\SO(V)= S(V)$. Since $K$ is a field of characteristic different from 2, it can be shown that for each symmetric pairing there exists a basis such that its matrix is diagonal.
If $K$ is an algebraically closed field, it can be shown that all symmetric pairings are equivalent. Then in this case we take the identity matrix $I_n$ as the matrix of the standard symmetric form. For such form, we will write $\GO_n(K)$, $\mbox{O}_n(K)$ and $\SO_n(K)$ instead $\GO(V)$, $\mbox{O}(V)$ and $\SO(V)$.

Let $\ell$ be an odd prime and $r$ be a positive integer. If $K$ is a finite field of order $\ell^r$ and $n$ is even, there are precisely two symmetric pairings on $V$ (up to equivalence), corresponding to the cases when the determinant of the matrix of the form is a square or non square of $K^\times$. We say that a symmetric pairing $\langle \cdot , \cdot \rangle$ has \emph{plus type} if its matrix is equivalent to 
\[
J_+ :=  \left( 
\begin{matrix}
0 &  & 1 \\
 & \iddots &  \\
1 &  & 0
\end{matrix}
\right) \in \mbox{M}_{n}, 
\]
otherwise it has \emph{minus type}.  As expected, $J_+$ will be the matrix of our standard symmetric pairing of plus type. For the minus type we will use the matrix $I_n$ when it is not equivalent to $J_+$ (this occurs if and only if $n\equiv 2 \mod 4$ and $\ell^r \equiv 3 \mod 4$). Otherwise, our standard symmetric pairing of minus type  will have matrix 
\[
J_- :=  \left( 
\begin{matrix}
\omega & &  &  \\
 & 1 &  &  \\
  &  & \ddots & \\
 &  &  & 1
\end{matrix}
\right) \in \mbox{M}_{n}, 
\]
where $\omega$ is a fixed primitive element of $K^\times$. Then for our standard symmetric pairing of plus type (resp. minus type) we will write $\GO^+_n(K)$, $\mbox{O}^+_n(K)$ and $\SO^+_n(K)$ (resp. $\GO^-_n(K)$, $\mbox{O}^-_n(K)$ and $\SO^-_n(K)$) instead $\GO(V)$, $\mbox{O}(V)$ and $\SO(V)$.

In contrast with the symplectic case where the projectivization $\PSp_n(K)$ of $\Sp_n(K)$ (with $n \geq 4$ and $K$ a finite field of odd characteristic) is a simple group, in the orthogonal case this does not happen. Then we need to define the \emph{quasisimple classical group} of the symmetric form $\langle \cdot , \cdot \rangle$ as $\Omega(V):=I'(V)$, the derived subgroup of $I(V)$. 
In particular, if $K$ is a finite field of odd characteristic, we denote by $\Omega^+_n(K)$ and $\Omega^-_n(K)$ the quasisimple orthogonal group of plus type and minus type respectively.  In particular, if $V$ is a vector space with a symmetric pairing over a finite field of odd characteristic and $n\geq 8$, it can be proved that $\POM(V)$ is a simple group (see Theorem 2.1.3 of \cite{KL90}). 

A useful tool, through this paper, will be to know the indices between the projectivizations of the symplectic and orthogonal groups defined above when $K$ is a finite field of odd characteristic. These indices are: $[\PGSp_n(K):\PSp_n(K)]=2$, $[\PGO^\pm_n(K):\PO^\pm_n(K)]=2$, $[\PO^\pm_n(K):\PSO^\pm_n(K)]=2$ and $[\PSO^\pm_n(K):\POM^\pm_n(K)]=a_\pm$, where the values of $a_+$ and $a_-$ are defined by the following conditions: $a_\pm\in \{ 1,2 \}$, $a_+ a_- = 2$, and $a_+ =2$ if and only if $n(\ell^r-1)/4$ is even.


\section{Polarized representations}\label{sec:3}

In this section we will review some facts about regular algebraic, essentially self-dual, cuspidal automorphic representations of $\GL_n(\A_\Q)$ and the Galois representations associated to them. Our main reference is \cite[$\S 2.1$]{BLGGT}. We refer the reader to loc. cit. for more details and references.

Let $\ell$ be a prime and $\iota: \overline{\Q}_\ell \cong \C$ be a fixed isomorphism. 
By a \emph{polarized} Galois representation of $G_{\Q}$ we will mean a pair $(\rho,\vartheta)$, where 
\[
\rho: G_\Q \longrightarrow \GL_n(\overline{\Q}_\ell) \quad \mbox{ and }\quad \vartheta :G_\Q \longrightarrow \overline{\Q}_\ell^\times  
\]
are continuous homomorphisms such that there is $\varepsilon \in \{ \pm 1 \}$ and a non-degenerate pairing $\langle \cdot , \cdot \rangle$ on $\overline{\Q}_\ell^n$ such that
\[
\langle x , y \rangle = \varepsilon \langle y , x \rangle \quad \mbox{ and } \quad 
\langle \rho (\sigma) x, \rho (c \sigma c)y \rangle = \vartheta(\sigma) \langle x, y \rangle
\]
for a complex conjugation $c$ and for all $x,y \in \overline{\Q}_\ell^n$ and all $\sigma \in G_\Q$. Note that $(\rho,\vartheta)$ is polarized Galois representation if and only if either $\vartheta(c) = - \varepsilon$ and $\rho$ factors through $\GSp_n(\overline{\Q}_\ell)$ with multiplier $\vartheta$ or $\vartheta(c) = \varepsilon$ and $\rho$ factors through $\GO_n(\overline{\Q}_\ell)$  with multiplier $\vartheta$. Finally, we say that $(\rho,\vartheta)$ is \emph{totally odd} if $\varepsilon =1$.

On the other hand, by a RAESDC (regular algebraic, essentially self-dual, cuspidal) automorphic representation of $\GL_n(\A_\Q)$ we mean a pair $(\pi, \mu)$ consisting of a cuspidal automorphic representation $\pi = \pi_\infty \otimes \pi_f$ of $\GL_n(\A_\Q)$ and a continuous character $\mu: \A_\Q^\times / \Q^\times \rightarrow \C^\times$ such that:
\begin{enumerate}
\item (regular algebraic) $\pi_\infty$ has the same infinitesimal character as an irreducible algebraic representation of $\GL_n$.
\item (essentially self-dual) $\pi \cong \pi ^{\vee} \otimes (\mu \circ \det)$.
\end{enumerate}

Note that $\mu$ is necessarily algebraic, i.e., there exist an integer $\alpha$ such that $\mu \vert_{(K^\times_\infty)^0} : x \mapsto x^\alpha$. Then, for a fixed prime $\ell$, we can attach to $\mu$ a unique continuous character 
\[
\rho_{\ell,\iota}(\mu): G_\Q \longrightarrow \overline{\Q}_\ell^\times
\]
such that for all $p \neq \ell$, we have $\iota \circ \rho_{\ell,\iota}(\mu) \vert _{W_{\Q_p}} \circ \Art_{\Q_p} = \mu_p$, where $\Art_{\Q_p}: \Q_p^\times \rightarrow W_{\Q_p}^{ab}$ is the Artin map normalized to send uniformizers to geometric Frobenius elements.

Let $a=(a_i) \in \Z^n$ such that $a_1 \geq \ldots \geq a_n$ and let $\Xi _{a}$ denotes the irreducible algebraic representation of $\GL_n$ with highest weight $a$. We say that a RAESDC automorphic representation $(\pi, \mu)$ of $\GL_n(\A_\Q)$ has \emph{weight} $a$ if $\pi_{\infty}$ has the same infinitesimal character as $\Xi _{a}^\vee$. 

Thanks to the work of Caraiani, Chenevier, Clozel, Harris, Kottwitz, Shin, Taylor and several others, we can associate compatible systems of Galois representations to RAESDC automorphic representations of $\GL_n(\A_\Q)$ as follows.

\begin{theo}\label{raesdc}
Let $(\pi, \mu)$ be a RAESDC automorphic representation of $\GL_n(\A_\Q)$ and $S$ be the finite set of primes $p$ such that $\pi_p$ is ramified. Then there is a compatible system of semi-simple Galois representations
\[
\rho_{\ell, \iota}(\pi) : G_\Q \longrightarrow \GL_n(\overline{\Q}_\ell)
\]
unramified outside $S \cup \{ \ell \}$ and such that the following properties are satisfied.
\begin{enumerate}
\item $(\rho_{\ell,\iota}(\pi), \chi_\ell^{1-n} \rho_{\ell,\iota}(\mu))$ is a totally odd polarized Galois representation, where $\chi_\ell$ denotes the $\ell$-adic cyclotomic character.
\item $\rho_{\ell, \iota} (\pi)$, restricted to a decomposition group at $\ell$, is de Rham and if $\ell \notin S$, it is crystalline.
\item The set of Hodge-Tate weights $\HT(\rho_{\ell,\iota}(\pi))$ of $\rho_{\ell, \iota}(\pi)$ is equal to
\[
\{ a_1 + (n-1), a_2+(n-2),\ldots , a_n \}.
\]
In particular, they are $n$ different numbers and they are independent of $\ell$.
\item Whether $p \nmid \ell$ or $p \vert \ell$, we have:
\[
\iota \WD(\rho_{\ell,\iota}(\pi) \vert_{\Q_p})^{F-ss} \cong \rec_p (\pi_p \otimes \vert \det \vert _p ^{(1-n)/2}).
\]
\end{enumerate}
\end{theo}

We will say that a compatible system $\mathcal{R} = \{ \rho_\ell \}_\ell$ of Galois representations $\rho_\ell:G_\Q \rightarrow \GL_n(\overline{\Q}_\ell)$ is \emph{totally odd polarizable} if there is a compatible system $\Theta = \{ \vartheta_\ell \}_\ell$ of characters $\vartheta_\ell : G_\Q \rightarrow \overline{\Q}_\ell$ such that $(\rho_\ell,\vartheta_\ell)$ is a totally odd polarized Galois representation for all $\ell$. In particular, the compatible system $\mathcal{R}(\pi) = \{  \rho_{\ell, \iota} (\pi) \}_\ell$ associated to a RAESDC automorphic representation $(\pi, \mu)$ of $\GL_n(\A_\Q)$ as in the previous theorem is totally odd polarizable with $\vartheta_\ell = \chi_\ell^{1-n} \rho_{\ell,\iota}(\mu)$.


\section{Regular representations}\label{sec:4}

Let $\overline{\rho} : G_\Q \rightarrow \GL_n(\overline{\F}_\ell)$ be a mod $\ell$ Galois representations. Recall that $\overline{\rho}$ is \emph{regular}, in the sense of \cite{ADW14}, if there exist an integer $s$ between $1$ and $n$, and for each $i=1, \ldots, s$ a set $A_i = \{ a_{i,1}, \ldots , a_{i,r_i} \}$ of natural numbers $0 \leq a_{i,j} \leq \ell-1$ of cardinality $r_i$, with $r_1 + \cdots + r_s = n$ (i.e., all the $a_{i,j}$ are distinct) such that if we denote by $B_i$ the matrix
\[
B_i \sim  
\left( \begin{array}{cccc}
\psi_{r_i}^{b_i} &   &  & 0 \\
 & \psi_{r_i}^{b_i \ell} &  & \\
 &  & \ddots & \\
0 &  & & \psi_{r_i}^{b_i \ell^{r_i-1}} \end{array} \right)
\]
with $\psi_{r_i}$ a fixed choice of a fundamental character of niveau $r_i$ and $b_i = a_{i,1} + a_{i,2}\ell + \cdots + a_{i,r_i} \ell^{r_i-1}$, then 
\[
\overline{\rho} \vert _{I_\ell} \sim  
\left( \begin{array}{ccc}
B_1  &  & * \\
  & \ddots & \\
0 & & B_s \end{array} \right).
\]
The elements of $A := A_1 \cup \cdots \cup A_s$ are called \emph{tame inertia weights} of $\overline{\rho}$. We say that $\overline{\rho}$ has \emph{tame inertia weights at most} $k$ if $A \subseteq \{ 0,1, \ldots, k\}$. 

Under the assumption of regularity and boundedness of tame inertia weights, we have the following useful result which was proved in Section 3 of \cite{ADW14}. 

\begin{lemm}\label{inert}
Let $n,k \in \N$, with $n$ even, and $\overline{\rho}: G_{\Q} \rightarrow \GL_n(\overline{\F}_\ell)$ be a Galois representation which is regular with tame inertia weights at most $k$. Assume that $\ell > kn! + 1$. Then all $n!$-th powers of the characters on the diagonal of $\overline{\rho} \vert _{I_\ell}$ are distinct.   
\end{lemm}

Let $K/\Q$ be a finite extension of degree $d$ and $\overline{\rho}_0:G_K \rightarrow \GL_m(\overline{\F}_\ell)$ be a Galois representation. Let $V_0$ be the $\overline{\F}_\ell$-vector space underlying $\overline{\rho}_0$.
The induced representation $\Ind_{G_K}^{G_\Q}\overline{\rho}_0$ of $\overline{\rho}_0$ from $G_K$ to $G_\Q$ is by definition the $\overline{\F}_\ell$-vector space $\Hom_{G_K} (G_\Q, V_0)$ defined as
\[
\{ \phi:G_\Q \rightarrow V_0 : \phi(\sigma \tau) = \overline{\rho}_0(\tau^{-1})\phi(\sigma) \mbox{ for all }\tau \in G_K \mbox{ and }\sigma \in G_\Q \},
\]
where $\sigma \in G_\Q$ acts on $\phi \in \Hom_{G_K} (G_\Q, V_0)$ by $\sigma \cdot \phi( \cdot ) = \phi(\sigma^{-1} \; \cdot \; )$. 
Let $\{ \gamma_1, \ldots, \gamma_d \}$ be a full set of representatives in $G_\Q$ of the left cosets in $G_\Q / G_K$. The map $\phi \mapsto \oplus_{i=1}^d \phi(\gamma_i)$ gives an isomorphism between $\Hom_{G_K} (G_\Q, V_0)$ and the direct sum $\bigoplus _{i=1}^d V_i$ (where each $V_i$ is isomorphic to $V_0$). Via this identification the action of $G_\Q$ on $\bigoplus _{i=1}^d V_i$ is given by
\[
( \Ind_{G_K}^{G_\Q} \overline{\rho}_0 ) (\sigma) ( \mathop{\oplus}_{i=1}^d v_i ) =  \mathop{\oplus}_{i=1}^d \overline{\rho}_0(\gamma_i^{-1} \sigma \gamma _{\sigma(i)})(v_{\sigma(i)}) , 
\]
where $\sigma^{-1}\gamma_i \in \gamma_{\sigma(i)} G_K$. Indeed, 
\[
\mathop{\oplus}_{i=1}^d \phi(\gamma_i)  \stackrel{\sigma}{\mapsto} \mathop{\oplus}_{i=1}^d \phi(\sigma^{-1} \gamma_{i}) = \mathop{\oplus}_{i=1}^d\overline{\rho}_0 (\gamma_i^{-1} \sigma \gamma_{\sigma(i)})(\phi(\gamma_{\sigma(i)})).
\] 
Then by using the previous lemma we can prove the following result about the ramification of induced representations.

\begin{coro}\label{indu}
Let $n,m,k \in \N$, $a \in \Z$ and $\ell > kn! + 1$ be a prime. Let $K/\Q$ be a finite extension of degree $d$ such that $dm=n$, $\overline{\rho}_0:G_K \rightarrow \GL_m(\overline{\F}_\ell)$ be a Galois representation and $\overline{\rho} = \Ind^{G_\Q}_{G_K} \overline{\rho}_0$. If  $\overline{\chi}_\ell^a \otimes \overline{\rho}$ is regular with tame inertia weights at most $k$, then $K/\Q$ does not ramify at $\ell$.
\end{coro}
\begin{proof}
The proof of this result is given in \cite[Proposition 3.4]{ADW14}.
\end{proof}

Keeping the same notation as in the previous paragraphs we define the \emph{tensor induced representation} $\TInd_{G_K}^{G_\Q}\overline{\rho}_0 : G_\Q \mapsto \GL(\bigotimes_{i=1}^d V_i)$ as:
\[
(\TInd_{G_K}^{G_\Q} \overline{\rho}_0)(\sigma) ( \mathop{\otimes}_{i=1}^d v_i ) =  \mathop{\otimes}_{i=1}^d \overline{\rho}_0(\gamma_i^{-1} \sigma \gamma _{\sigma(i)})(v_{\sigma(i)}).
\]
Note that for all $\sigma \in G_\Q$ the map $\gamma_i \mapsto \gamma_{\sigma(i)}$ is a permutation of $\{ 1, \ldots d\}$ which is trivial if and only if $\sigma \in G_{\widetilde{K}}$, where $\widetilde{K}$ denotes the Galois closure of $K/\Q$. Then for each $\sigma \in G_{\widetilde{K}}$ we have that 
\[
(\TInd_{G_K}^{G_\Q} \overline{\rho}_0)(\sigma) =  \mathop{\otimes}_{i=1}^d \overline{\rho}_0(\gamma_i^{-1} \sigma \gamma_{i}).
\]

\begin{coro}\label{tindu}
Let $n,m,k \in \N$, $a \in \Z$ and $\ell > kn! + 1$ be a prime. Let $K/\Q$ be a finite extension of degree $d$ such that $m^d=n$, $\overline{\rho}_0:G_K \rightarrow \GL_m(\overline{\F}_\ell)$ be a Galois representation and $\overline{\rho} = \TInd^{G_\Q}_{G_K} \overline{\rho}_0$. If $\overline{\chi}_\ell^a \otimes \overline{\rho}$ is regular with tame inertia weights at most $k$, then $K/\Q$ does not ramify at $\ell$.
\end{coro}
\begin{proof}
The proof is adapted from \cite[Proposition 3.4]{ADW14}. Then we start by assuming that $K/\Q$ ramifies at $\ell$ and we arrive to a contradiction.

Let $V_0$ be the $\overline{\F}_\ell$-vector space underlying $\overline{\rho}_0$. For all $\gamma\in G_\Q$, we define 
\[
^\gamma \overline{\rho}_0:G_{\gamma(K)} \longrightarrow \GL(V_0)
\]
by $^\gamma \overline{\rho}_0(\sigma)=\overline{\rho}_0(\gamma^{-1} \sigma \gamma)$. 
Let $\Lambda$ be a fixed prime of the Galois closure $\widetilde{K}$ of $\Q$ above $\ell$, $I_\Lambda \subseteq G_{\widetilde{K}}$ be the inertia group at the prime $\Lambda$ and $P_{\ell} \subseteq I_\ell$ be the wild inertia group at $\ell$.
Let $\sigma \in I_\ell$ and $\tau \in I_\Lambda$. Since $I_\ell / P_{\ell}$ is pro-cyclic, we have that the commutator $\sigma^{-1} \tau \sigma \tau^{-1}$ belongs to $P_{\ell}$, then ${^\gamma \overline{\rho}_0(\sigma^{-1} \tau \sigma \tau^{-1})} \in {^\gamma \overline{\rho}_0 (P_{\ell})}$. Moreover, since $I_\Lambda \subseteq I_\ell$ is normal, we have that $\sigma^{-1} \tau \sigma \in I_\Lambda \subseteq G_{\widetilde{K}} \subseteq G_{\gamma(K)}$, so $^\gamma \overline{\rho}_0 (\sigma^{-1} \tau \sigma)$ and $^\gamma\overline{\rho}_0(\tau^{-1})$ belong to ${^\gamma \overline{\rho}_0 (I_{\Lambda})}$.
Now, choose a basis such that the image of $I_{\Lambda}$ is contained in the upper triangular matrices. Then, $^\gamma \overline{\rho}_0 (\sigma^{-1} \tau \sigma)$ and $^\gamma \overline{\rho}_0(\tau^{-1})$ are upper triangular matrices and their product $^\gamma \overline{\rho}_0 (\sigma^{-1} \tau \sigma) ^\gamma \overline{\rho}_0(\tau^{-1}) = {^\gamma \overline{\rho}_0(\sigma^{-1} \tau \sigma \tau^{-1})}$, which is contained in the image of the wild inertia, is a upper triangular matrix such that the entries on its main diagonal are all 1. 
Thus, the elements on the main diagonal of $^\gamma \overline{\rho}_0 (\sigma^{-1} \tau \sigma)$ and $^\gamma \overline{\rho}_0(\tau^{-1})$ are reciprocal, so $^\gamma \overline{\rho}_0 (\sigma^{-1} \tau \sigma)$ and $^\gamma \overline{\rho}_0(\tau)$ have exactly the same eigenvalues.

As we are assuming that $\tilde{K}/\Q$ ramifies at $\ell$, we can pick $\sigma \in I_\ell \setminus G_{\widetilde{K}}$, and as $ \widetilde{K} = \prod _{\gamma \in G_\Q} \gamma(K)$, there exists some $\gamma \in G_\Q$ such that $\sigma \notin G_{\gamma(K)}$. This implies that $\overline{\rho}_0 (\sigma \gamma)(V_0) \cap \overline{\rho}_0(\gamma)(V_0)=0$.
Let $\{\gamma_1, \ldots, \gamma_d \}$ be a full set of left-coset representatives of $G_K$ in $G_\Q$ with $\gamma_1 = \gamma$ and $\gamma_2 = \sigma \gamma$. 
As $\tau \in I_\Lambda \subseteq G_{\widetilde{K}}$, we have that 
\[
\overline{\rho}(\tau) = \mathop{\otimes}^d_{i=1} {^{\gamma_i} \overline{\rho}_0(\tau)},
\]
where one factor is $^\gamma \overline{\rho}_0(\tau)$ and another factor is $^{\sigma \gamma} \overline{\rho}_0(\tau) = {^\gamma \overline{\rho}_0(\sigma^{-1} \tau \sigma)}$. 
Let $\mu_1, \ldots , \mu_m$ be the eigenvalues of $^\gamma \overline{\rho}_0(\tau)$ and $\mu'_1 , \ldots , \mu'_m$ be those of $^{\gamma} \overline{\rho}_0(\sigma^{-1} \tau \sigma)$. Then the eigenvalues of $^{\gamma} \overline{\rho}_0(\tau) \otimes {^\gamma \overline{\rho}_0(\sigma^{-1} \tau \sigma)}$ are $\{\mu_i \mu'_j: i,j=1,\ldots ,m \}$.
On the other hand, by Lemma \ref{inert} we have that the $n!$-powers of the characters on the diagonal of $\overline{\chi}_\ell^a \otimes \overline{\rho} \vert _{I_\ell}$ are all different, which implies that the characters on the diagonal of $\overline{\rho} \vert_{I_\Lambda}$ are all different. Thus $^{\gamma} \overline{\rho}_0(\tau)$ and $^{\gamma} \overline{\rho}_0(\sigma^{-1} \tau \sigma)$ cannot have the same eigenvalues for all $\tau \in I_\Lambda$. Then we have a contradiction.
\end{proof}


\section{Maximally induced representations}\label{sec:5}

Let $n=2m$ be an even integer and $p,q > n$ be distinct odd primes such that the order of $q$ mod $p$ is $n$. Denote by $\Q_{q^n}$ the unique unramified extension of $\Q_q$ of degree $n$ and recall that $\Q^{\times}_{q^n} \simeq \mu_{q^n-1} \times U_1 \times q^{\Z}$, where $\mu_{q^n-1}$ is the group of $(q^n-1)$-th roots of unity and $U_1$ is the group of 1-units. 

\begin{defi}\label{char}
Let $p$, $q$ be primes as above and $\ell$ be a prime distinct from $p$ and $q$. A character 
\[
\chi_q : \Q^{\times}_{q^n} \longrightarrow \overline{\Q}^{\times}_\ell
\] 
is of $S$-\emph{type} (resp. of $O$-\emph{type}) \emph{and order} $p$ if satisfies the following conditions:
\begin{enumerate}
\item $\chi_q$ has order $2p$ (resp. $p$),
\item $\chi_q \vert _{\mu_{q^n-1} \times U_1}$ has order $p$, and
\item $\chi_q (q) = -1$ (resp. $\chi_q (q) = 1$). 
\end{enumerate}
\end{defi}

Note that a character of $S$-type or $O$-type and order $p$ is tame. Recall that a character of $\Q^\times_{q^n}$ is \emph{tame} if it is trivial on $U_1$. By local class field theory we can regard $\chi_q$ as a character (which by abuse of notation we call also $\chi_q$) of $G_{\Q_{q^n}}$ or of $W_{\Q_{q^n}}$. Then we can define the Galois representation 
\[
\rho_q := \Ind^{G_{\Q_q}}_{G_{\Q_{q^n}}}(\chi_q).
\] 

\begin{lemm}\label{irred}
Let $\chi_q$ be a character of $S$-type (resp. of $O$-type) and order $p$. Then the representation $\rho_q$ is irreducible and symplectic (resp. orthogonal) in the sense that it can be conjugated to take values in   $\Sp_n(\overline{\Q}_\ell)$ (resp. $\SO_n(\overline{\Q}_\ell)$).
 Moreover, if $\alpha: G_{\Q_q} \rightarrow \overline{\Q}^\times_\ell$ is an unramified character, then the residual representation $\overline{\rho}_q \otimes \overline{\alpha}$ is also irreducible. 
\end{lemm}

\begin{proof}
As the order of $\chi_q$ restricted to the inertia group at $q$ is $p$ and the order of $q$ mod $p$ is $n$, the characters $\chi_q, \chi_q^q, \ldots, \chi_q^{q^{n-1}}$ are all distinct. Then $\rho_q$ is irreducible. 
Moreover, since $\chi_q$ is tame and $\chi_q \vert _{q^\Z}$ is of order 2 or trivial (according to whether $\chi_q$ is of $S$-type or $O$-type and order $p$), Theorem 1 of \cite{moy} implies that $\rho_q$ is symplectic or orthogonal. 

Let $\overline{\chi}_q$ (resp. $\overline{\alpha}$) be the composite of $\chi_q$ (resp. $\alpha$) with the projection $\overline{\Z}_\ell \twoheadrightarrow \overline{\F}_\ell$. Note that the image of the reduction of $\rho_q$ in $\GL_n(\overline{\F}_\ell)$ is $\Ind^{G_{\Q_q}}_{G_{\Q_{q^n}}} (\overline{\chi}_q)$, which is an irreducible representation.
Since $\overline{\alpha}$ is unramified, the order of the restriction of $\overline{\chi}_q \otimes (\overline{\alpha} \vert _{\Q_{q^n}})$ to the inertia group at $q$ is $p$. Then as the order of $q$ mod $p$ is $n$, the $n$ characters $(\overline{\chi}_q \otimes (\overline{\alpha} \vert _{\Q_{q^n}})), (\overline{\chi}_q \otimes (\overline{\alpha} \vert _{\Q_{q^n}}))^q,\ldots, (\overline{\chi}_q \otimes (\overline{\alpha} \vert _{\Q_{q^n}}))^{q^{n-1}}$ are different which implies the irreducibility of $\overline{\rho}_q \otimes \overline{\alpha} = \Ind^{G_{\Q_q}}_{G_{\Q_{q^2}}}(\overline{\chi}_q \otimes \overline{\alpha} \vert_{G_{\Q_{q^n}}})$.
\end{proof}

\begin{defi}\label{minduced}
Let $p$, $q$ and $\ell$ be primes as above. We say that a Galois representation 
\[
\rho: G_\Q \longrightarrow \GL_n(\overline{\Q}_\ell)
\]
is \emph{maximally induced of }$S$-\emph{type} (resp. \emph{of} $O$-\emph{type}) \emph{at} $q$ \emph{of order} $p$ if the restriction of $\rho$ to a decomposition group at $q$ is equivalent to $\Ind^{G_{\Q_q}}_{G_{\Q_{q^n}}} (\chi_q) \otimes\alpha$, where $\chi_q$ is a character of $S$-type (resp. of $O$-type) and order $p$ and $\alpha : G_{\Q_q} \rightarrow \overline{\Q}^\times_\ell$ is an unramified character.
\end{defi}

Note that if $\chi_q$ is a character of $S$-type (resp. of $O$-type) and order $p$, then the image $G_q$ of $\overline{\rho}_q$  is homomorphic to a non-abelian extension of $\Z / n\Z$ by $\Z/2p\Z$ (resp. by $\Z/p\Z$) such that $\Z/n\Z$ acts faithfully on $\Z/p\Z \subseteq \Z/2p\Z$. Then we have the following result.

\begin{lemm}\label{npg} 
Let $\ell$ be a prime different from $p$ and $q$. Then every irreducible representation of $G_q$ over $\overline{\F}_\ell$ has dimension 1 or dimension at least $n$. 
\end{lemm}

\begin{proof}
The proof is adapted from Lemma 2.1 of \cite{KLS} where the case $\chi_q$ of $S$-type and order $p$ is dealt. Then we can assume that $\chi_q$ is of $O$-type and order $p$. In this case we have the following exact sequence
\[
0 \longrightarrow \Z/p\Z \longrightarrow G_q \longrightarrow \Z / n\Z \longrightarrow 0
\]
with $\Z/n\Z$ acting on $\Z/p\Z$ faithfully. Note that the restriction of any irreducible representation of $G_q$ to $\Z/p\Z$ is a direct sum of characters because $\ell$ is different from $p$. If every character is trivial, then the original representation factors through $\Z/n\Z$ which is abelian, so 1-dimensional.

Otherwise, a non-trivial character $\chi$ of $\Z/p\Z$ appears. Then, every character obtained by composing $\chi$ with an automorphism of $\Z/p\Z$ coming from the action of $\Z/n\Z$ likewise appears. As there are $n$ such characters, the original representation must have degree at least $n$.   
\end{proof}

\begin{lemm}\label{tens} 
Let $\ell$ be a prime different from $p$ and $q$. Then every faithful $n$-dimensional representation of $G_q$ over $\overline{\F}_\ell$ is tensor-indecomposable.
\end{lemm}
\begin{proof}
Assume that there exists a faithful $n$-dimensional representation of $G_q$ which is tensor-decomposable.
Then it can be written as a tensor product $\rho_1 \otimes \cdots \otimes \rho_h$ of irreducible representations of $G_q$ over $\overline{\F}_\ell$ of dimension greater than $1$ but smaller than $n$. So by Lemma \ref{npg} we obtain a contradiction.
\end{proof}


\section{Study of the images I (geometric cases)}\label{sec:6}

In this section we will give a criterion to know when a totally odd polarizable compatible system $\mathcal{R} = \{ \rho_\ell \}_\ell$ of $n$-dimensional Galois representations, which are maximally induced of $S$-type or $O$-type at $q$ of order $p$ for an ``appropriate" couple of primes $(p, q)$, is such that the image of $\overline{\rho}_\ell$ cannot be contained in a maximal subgroup of ``geometric type" of an $n$-dimensional symplectic or orthogonal group for almost all $\ell$. 

In order to be more precise we need to explain what we mean by a maximal subgroup of geometric type of a finite almost simple classical group.
Recall that an \emph{almost simple group} is a group $H$ such that $S \unlhd H \leqslant \Aut(S)$ for some non-abelian simple group $S$.
Note that as the Galois representations $\rho_\ell$ in $\mathcal{R}$ are totally odd polarized, we can ensure that the image of $\rho_\ell$ lies in an orthogonal or symplectic group, then it is enough for our purposes to study the maximal subgroups of $\GSp_n(\F_{\ell^r})$ and $\GO^\pm_n(\F_{\ell^r})$.  Such subgroups were classified essentially by Aschbacher in \cite{As84}, but see also Main Theorem in Chapter 3 of \cite{KL90} and Main Theorem 2.1.1 of \cite{BHR13} for a more precise formulation of such classification. 

Let $\ell$ be an odd prime and $n,r \in \N$ with $n$ even. Let $G$ be a maximal subgroup of $\GSp_n(\F_{\ell^r})$ (resp. of $\GO^\pm_n(\F_{\ell^r})$) which does not contain $\Sp_n(\F_{\ell^r})$ (resp. $\Omega^\pm_n(\F_{\ell^r})$). If $n \geq 6$ (resp. $n \geq 10$), at least one of the following holds:
\begin{enumerate}
\item $G$ stabilizes a totally singular or a non-singular subspace; 
\item $G$ stabilizes a decomposition $V = \oplus_{i=1}^t V_i$, $\dim(V_i) = n / t$;
\item $G$ stabilizes an extension field of $\F_{\ell^r}$ of prime index dividing $n$;
\item $G$ stabilizes a tensor product decomposition $V = V_1 \otimes V_2$;
\item $G$ stabilizes a decomposition $V = \otimes_{i=1}^t V_i$, $\dim(V_i)=a$, $n=a^t$; 
\item $G$ normalizes an extraspecial or a symplectic type group; 
\item $G$ stabilizes a subfield of $\F_{\ell^r}$ of prime index; or
\item $G$ is of class $\mathcal{S}$.
\end{enumerate}
 
Recall that a maximal subgroup $G$ of $\GSp_n(\F_{\ell^r})$ (resp. of $\GO^\pm_n(\F_{\ell^r})$) is called of \emph{of class} $\mathcal{S}$ if all of the following holds:
\begin{enumerate}
\item $\mbox{P}G$ is almost simple;
\item $G^\infty$ acts absolutely irreducible;
\item  $G^\infty$ is not conjugated to a group defined over a proper subfield of $\F_{\ell^r}$;
\item $G$ does not contain $\Sp_{n}(\F_{\ell^r})$ (resp. $\Omega_n^\pm(\F_{\ell^r})$).
\end{enumerate}
  
We refer the reader to the subsection of Notation, in the Introduction of this paper, for the definition of $G^{\infty}$. 
  
\begin{rema} Note that the definition of a group of class $\mathcal{S}$ is slightly weaker that the classical definition (see Definition 2.1.3 of \cite{BHR13} and $\S 1.2$ of \cite{KL90}). However, as we assume $\ell$ odd, according to Table 4.8.A of \cite{KL90}, both definitions are equivalent. 
\end{rema}  

  \begin{defi}\label{geot}
We will say that a maximal subgroup $G$ of $\GSp_n(\F_{\ell^r})$ or $\GO^\pm_n(\F_{\ell^r})$ is of \emph{geometric type} if this lies in one of the first six cases of Aschbacher's classification.
\end{defi}
 
We remark that this definition is weaker than the classical definition which also considers the maximal subgroups lying in case vii) (of the previous classification) as of geometric type. However, it is enough for our purposes because, as we will see in Section \ref{sec:8}, we are not interested in excluding these kinds of groups. 

We refer the reader to $\S 2.2$ of \cite{BHR13} and Chapter 4 of \cite{KL90} for the rest of relevant definitions concerning to Aschbacher's classification.
In order to be able to use the previous classification of maximal subgroups of $\GSp_n(\F_{\ell^r})$ (resp. of $\GO^\pm_n(\F_{\ell^r})$), henceforth, we will assume that $\ell$ is a odd prime and $n\geq 6$ (resp. $n \geq 10$).

Finally, before to state our main result, we give a basic lemma which explains what we mean by an appropriate couple of primes.

\begin{lemm}\label{prim}
Let $k,n,N \in \N$ with $n$ even and $M$ be an integer greater than $17$, $n$, $N$, $kn!+1$, and all primes dividing $2\prod_{i=1}^{m}(2^{2i} -1)$ if $n=2^m$ for some $m \in \N$. Let $L_0$ be the compositum of all number fields of degree smaller that or equal to $n!$ which are ramified at most at the primes smaller than or equal to $M$. Then we can choose two different primes $p$ and $q$ such that: 
\begin{enumerate}
\item $p \equiv 1 \mod n$,
\item $p$ and $q$ are greater than $M$,
\item $q$ splits completely in $L_0$, and
\item $q^{n/2} \equiv -1 \mod p$. 
\end{enumerate}
\end{lemm}
\begin{proof}
First, choose a prime $p$ greater than $M$ and such that $p \equiv 1 \mod n$. Then Chevotarev's Density Theorem allows us to choose a prime $q>M$ (from a set of positive density) which splits completely in $L_0$ and such that $q^{n/2}\equiv -1 \mod p$. 
\end{proof}

The main result of this paper is the following:

\begin{theo}\label{rt}
Let $k$, $n$, $N$, $M$, $p$, $q$ and $L_0$ as in Lemma \ref{prim}. Let $\mathcal{R} = \{ \rho_\ell \}_\ell$ be a totally odd polarizable compatible system of Galois representations $\rho_\ell: G_\Q \rightarrow \GL_n(\overline{\Q}_\ell)$ such that for every prime $\ell$, $\rho_\ell$ ramifies only at the primes dividing $Nq\ell$. 
Assume that for every $\ell>kn!+1$ a twist of $\overline{\rho}_\ell$ by some power of the cyclotomic character is regular with tame inertia weights at most $k$ and  that for all $\ell \neq p,q$, $\rho_\ell$ is maximally induced of $S$-type (resp. of $O$-type) at $q$ of order $p$.
Then for all primes $\ell$ different from $2$, $p$ and $q$, the image of $\overline{\rho}_\ell$ cannot be contained in a maximal subgroup of $\GSp_n(\F_{\ell^r})$  (resp. of $\GO^\pm_n(\F_{\ell^r})$) of geometric type.
\end{theo}

\begin{rema}\label{remi} In the orthogonal case we assume $n \geq 10$ because: $\Omega^{\pm}_2(\F_{\ell^r}) \cong \Z_{(\ell^r \mp 1)/(2,\ell^r-1)}$, $\Omega^+_4(\F_{\ell^r}) \cong \SL_2(\F_{\ell^r}) \circ \SL_2(\F_{\ell^r})$, $\Omega^-_4 (\F_{\ell^r}) \cong \PSL_2(\F_{\ell^{2r}})$, $\Omega_6^+(\F_{\ell^r}) \cong \SL_4(\F_{\ell^r})/ \langle I_4 \rangle$, $\Omega_6^-(\F_{\ell^r}) \cong \SU_4(\F_{\ell^r})/ \langle I_4 \rangle$ (see Proposition 2.9.1 of \cite{KL90} and $\S 1.10$ of \cite{BHR13}) and $\Aut(\POM^+_8(\F_{\ell^r})) \neq \POMO^+_8(\F_{\ell^r})$, where $\Gamma \mbox{O}^+_8(\F_{\ell^r})$ denotes the group of all semi-isometries of $\F_{\ell^r}^8$ with the standard symmetric pairing of positive type (see the remarks after Theorem 1.2.1 of \cite{KL90} and Tables 8.50 and 8.51 of \cite{BHR13}). Then Aschbacher's classification does not apply in these cases. In the symplectic case we assume $n \geq 6$ because the cases $\PGL_2(\F_{\ell^r})$ and $\PGSp_4(\F_{\ell^r})$ have been studied in \cite{DW11} and \cite{DZ16} respectively.
\end{rema}

Now, we are ready to give the proof of Theorem \ref{rt}, which will be given by showing that $G_\ell= \mbox{Im}(\overline{\rho}_\ell)$ is not contained in any subgroup lying in cases i)-vi) of Aschbacher's classification.
In the rest of this section $\Gamma_q$ will denote the image of $\overline{\rho}_\ell \vert_{D_q}$.


\subsection{Reducible cases} \label{sec:61}

Let $V$ be the space underlying $\overline{\rho}_\ell$. Suppose that $G_\ell$ is a subgroup of a maximal subgroup in case i) of Aschbacher's classification, then $G_\ell$ stabilizes a proper non-zero totally singular or a non-degenerate subspace of $V$. Therefore $G_\ell$ does not act irreducibly on $V$. But, according to Lemma \ref{irred}, if $\ell \neq p,q$, $G_\ell$ acts irreducibly on $V$. Hence, if we assume $\ell$ different from $p$ and $q$, $\overline{\rho}_\ell$ cannot be reducible.


\subsection{Imprimitive and field extension cases}\label{sec:62}

If $G_\ell$ is an irreducible subgroup of a maximal subgroup in cases ii) and iii) of Aschbacher's classification, then there exists a normal subgroup $H_\ell$ of index at most $n!$ of $G_\ell$ such that
\[
1 \longrightarrow H_\ell \longrightarrow G_\ell \longrightarrow U_\ell \longrightarrow 1,
\]
where $U_\ell$ is isomorphic to a subgroup of $S_t$, the symmetric group of $t$ elements ($1 < t \leq n$), and $H_\ell$ is reducible (not necessarily over $\F_{\ell^r}$). See $\S 2.2.2$ and $\S 2.2.3$ of \cite{BHR13}.

Let $L$ be the Galois extension of $\Q$ corresponding to $H_\ell$. Note that as $\overline{\rho}_\ell(I_q)$ has order $p$ and $(p,n!)=1$, $L/\Q$ is unramified at $q$. Then from the ramification of $\overline{\rho}_\ell$ we have that $L$ is unramified outside $\{ \ell, p_1, \ldots, p_w \}$, where $p_1, \ldots, p_w$ are the primes smaller than or equal to the bound $M$.
If $\ell > kn!+1$ and different from $p$ and $q$, it follows from Corollary \ref{indu} that $L$ is unramified at $\ell$ and if $\ell \leq kn!+1$, then $\ell \in \{p_1, \ldots, p_w \}$. Then in both cases  $L$ is contained in $L_0$. This implies that $q$ splits completely in $L$ and therefore $\Gamma_q$ is contained in $H_\ell$ for all primes $\ell$ different from $p$ and $q$, and according to Lemma \ref{irred}, $H_\ell$ should be absolutely irreducible. Then we have a contradiction.
 

\subsection{Tensor product cases}\label{sec:63}

Now assume that $G_\ell$ is a subgroup of a maximal subgroup in the case iv) of Aschbacher's classification, then the representation $\overline{\rho}_\ell$ can be written as a tensor product $\overline{\rho}_1 \otimes \overline{\rho}_2$ of two representations $\overline{\rho}_1$ and $\overline{\rho}_2$ with $\dim(\overline{\rho}_i) < n$ for $i=1,2$. Then as $G_\ell$ contain $\Gamma_q$ for all primes $\ell$ different from $p$ and $q$, we have that the restriction of $\overline{\rho}_\ell$ to $D_q$ arises from the tensor product of two representations over $\overline{\F}_\ell$ of dimension greater than 1 and smaller than $n$. But, by Lemma \ref{tens} we have that this restriction is tensor-indecomposable. Then we have a contradiction.


\subsection{Tensor induced cases}\label{sec:64}

Similarly to $\S$ \ref{sec:62}, if $G_\ell$ is an irreducible subgroup of a maximal subgroup in case v) of Aschbacher's classification, then there exists a normal subgroup $T_\ell$ of index at most $n!$ of $G_\ell$ such that
\[
1 \longrightarrow T_\ell \longrightarrow G_\ell \longrightarrow U_\ell \longrightarrow 1,
\]
where $U_\ell$ is isomorphic to a subgroup of $S_t$, $1 < t \leq n$, and $T_\ell$ is tensor-decomposable. See $\S 2.2.7$ of \cite{BHR13}.

Let $L$ be the Galois extension of $\Q$ corresponding to $T_\ell$. From the ramification of $\overline{\rho}_\ell$ we have that $L$ is unramified outside $\{ \ell, q, p_1, \ldots, p_w \}$, where $p_1, \ldots, p_w$ are the primes smaller than or equal to the bound $M$. Note that, since $\overline{\rho}_\ell(I_q)$ has order $p$ and $(p,n!)=1$, $L/\Q$ is unramified at $q$. Moreover, if $\ell > kn!+1$ and different from $p$ and $q$, it follows from Corollary \ref{tindu} that $L$ is unramified at $\ell$ and if $\ell \leq kn!+1$, then $\ell \in \{p_1, \ldots, p_w \}$. Thus $L$ is contained in $L_0$ which implies that $q$ splits completely in $L$ and therefore $\Gamma_q$ is contained in $T_\ell$ for all primes $\ell$ different from $p$ and $q$. So Lemma \ref{tens} implies that $T_\ell$ is tensor-indescomposable, so we have a contradiction.
 

\subsection{Extraspecial cases} \label{sec:65}

Recall that a $2$-group $R$ is called \emph{extraspecial} if its center $Z(R)$ is cyclic of order $2$ and the quotient $R/Z(R)$ is a non-trivial elementary abelian $2$-group. 
For any integer $m>0$ there are two types of extraspecial groups of order $2^{1+2m}$. We write $2^{1+2m}_+$ for the extraspecial group of order $2^{1+2m}$ which is isomorphic to a central product of $m$ copies of $D_8$ and we write $2_-^{1+2m}$ for the extraspecial group of the same order but  that is isomorphic to a central product of $m-1$ copies of $D_8$ and one of $Q_8$. 

Now suppose that $G_\ell$ is a subgroup of a maximal subgroup in case vi) of Aschbacher's classification. First, observe that according to Table 3.5.F of \cite{KL90} there are no subgroups of $\GO^-_n(\F_{\ell^r})$ belonging to this case, then we can assume that $G_\ell$ is either a subgroup of $\GO^+_n(\F_{\ell^r})$ or a subgroup of $\GSp_n(\F_{\ell^r})$. Moreover, according to Table 3.5.E of loc. cit.,  $G_\ell$ lies in this case only if $n=2^m$, $r=1$, $\ell \geq 3$; and $G_\ell \subseteq N_{\GO^+_n(\F_{\ell^r})}(R)$ or $G_\ell \subseteq N_{\GSp_n(\F_{\ell^r})}(R)$, where $R$ is an absolutely irreducible $2$-group of type $2_-^{1+2m}$ or of type $2_+^{1+2m}$.
From (4.6.1) of \cite{KL90} we have that the projective image $\mbox{P}G_\ell$ of $G_\ell$ in $\PGO^+_n(\F_{\ell^r})$ or in $\PGSp_n(\F_{\ell^r})$ is isomorphic to a subgroup of $C_{\Aut(R)}(Z(R))$. Then from Table 4.6.A of loc. cit., we have that $\mbox{P}G_\ell \subseteq 2^{2m}. \mbox{O}^+_{2m}(\F_2)$ (of order $2^{m^2+m+1}(2^m-1) \prod_{i=1}^{m-1} (2^{2i}-1)$) or $\mbox{P}G_\ell \subseteq 2^{2m}. \mbox{O}^-_{2m}(\F_2)$ (of order $2^{m^2+m+1}(2^m+1) \prod_{i=1}^{m-1} (2^{2i}-1)$).
Then $G_\ell$ cannot be a subgroup of a maximal subgroup in case vi) of Aschbacher's classification because we have chosen $p$ not dividing $2 \prod_{i=1}^{m}(2^{2i}-1)$.


\subsection{Conclusion}\label{sec:66}

Having gone through the cases i)-vi) of Aschbacher's classification, we can conclude that, for all odd primes $\ell$ different from $p$ and $q$, the image of $\overline{\rho}_\ell$ cannot be contained in a maximal subgroup of geometric type.


\section{Existence of compatible systems with prescribed local conditions}\label{sec:7}

The goal of this section is to prove the existence of totally odd polarized compatible systems of Galois representations satisfying Theorem \ref{rt} via automorphic representations. We closely follow the arguments in \cite{ADSW14}.

We start with an existence theorem of automorphic representations for split classical groups over $\Q$ which are based on the principle that the local components of automorphic representations at a fixed prime are equidistributed in the unitary dual of a reductive group according to an appropriate measure. More precisely: 

\begin{theo}\label{shins} 
Let $G$ be a split classical group over $\Q$ such that $G(\R)$ has discrete series.
Let $S$ be a finite set of primes and $\hat{U}_p$ be a prescribable subset for each $p \in S$. Then there exist cuspidal automorphic representations $\tau$ of $G(\A_\Q)$ such that 
\begin{enumerate}
\item $\tau_p \in \hat{U}_p$ for all $p\in S$,
\item $\tau$ is unramified at all finite places away from $S$, and
\item $\tau_\infty$ is a discrete series whose infinitesimal character is sufficiently regular.
\end{enumerate}
\end{theo}
For the definiton of prescribable subsets and suficiently regular characters we refer to $\S 3.1$  and $\S 3.2$ of \cite{ADSW14} respectively.
\begin{proof}
This result is the analogue of Theorem 5.8 of \cite{Shi12} except that here we are not assuming that the center of $G$ is trivial.  However, in this case we can fix a central character and apply the trace formula with fixed central character as in Section 3 of \cite{Bin} to deduce the exact analogue of Theorem 4.11 and Corolary 4.12 of \cite{Shi12}. Then our result can be deduced in the same way as in \cite[Theorem 5.8]{Shi12}.
\end{proof}

On the other hand, Arthur has recently classified local and global automorphic representations of symplectic and special orthogonal groups via twisted endoscopy relative to general linear groups \cite{Ar12}. 
For our purpose it suffices to consider the split special orthogonal groups: 
\begin{itemize}
\item $\SO_{2m+1}$ with the natural embedding $\xi : \Sp_{2m}(\C) \rightarrow \GL_{2m}(\C)$, $m \in \N$, and 
\item $\SO_{2m}$ with the natural embedding $\xi' : \SO_{2m}(\C) \rightarrow \GL_{2m}(\C)$, $m\in \N$ even\footnote{This last restriction is because $\SO_{2m}(\R)$ has no discrete series if $m$ is odd.}.
\end{itemize}

Recall that if $v$ is a finite (resp. archimedean) place, $W'_{\Q_v} = W_{\Q_v} \times \SL_2(\C)$ (resp. $W'_{\Q_v} = W_{\Q_v}$) denotes the Weil-Deligne group of $\Q_v$. We will say that an $L$-parameter $\phi_v : W'_{F_v} \rightarrow \GL_{2m}(\C)$ is of \emph{symplectic type} (resp. \emph{orthogonal type}) if it preserves a suitable alternating (resp. symmetric) form on the $2m$-dimensional space, or equivalently, if $\phi_v$  factors through $\xi$ (resp. $\xi'$) possibly after conjugation by an element of $\GL_{2m}(\C)$.

For each local $L$-parameter $\phi_v: W'_{\Q_v} \rightarrow \Sp_{2m}(\C)$ (resp. $\phi'_v: W'_{\Q_v} \rightarrow \SO_{2m}(\C)$), Arthur \cite[Theorem 1.5.1]{Ar12} associates an $L$-packet $\Pi_{\phi_v}(\SO_{2m+1}(\Q_v))$ (resp. $\Pi_{\phi'_v}(\SO_{2m}(\Q_v))$) consisting of finitely many irreducible representations of $\SO_{2m+1}(\Q_v)$ (resp. $\SO_{2m}(\Q_v)$). Moreover, each irreducible representation belongs to the $L$-packet for a unique parameter up to equivalence. If $\phi_v$ (resp. $\phi'_v$) has finite centralizer group in $\Sp_{2m}(\C)$ (resp. $\SO_{2m}(\C)$) so that it is a discrete parameter, then $\Pi_{\phi_v}(\SO_{2m+1}(\Q_v))$ (resp. $\Pi_{\phi'_v}(\SO_{2m}(\Q_v))$) consists only of discrete series. 

Let $\tau$ (resp. $\tau'$) be a discrete automorphic representation of $\SO_{2m+1}(\A_\Q)$ (resp. $\SO_{2m}(\A_\Q)$). Arthur \cite[Theorem 6.1]{Ar14} shows the existence of a self-dual isobaric automorphic representation $\pi$ (resp. $\pi'$) of $\GL_{2m}(\A_\Q)$ which is a functorial lift of $\tau$ (resp. $\tau'$) along the embedding $\xi : \Sp_{2m}(\C) \rightarrow \GL_{2m}(\C)$ (resp. $\xi' : \SO_{2m}(\C) \rightarrow \GL_{2m}(\C)$). 
In the generic case in the sense of Arthur (i.e., when the $\SL_2$-factor in the global $A$-parameter for $\tau$ (resp. $\tau'$) has trivial image) this means that for the unique $\tau_v$ (resp. $\tau'_v$) such that $\tau_v \in \Pi_{\phi_v}(\SO_{2m+1}(\Q_v))$ (resp. $\tau'_v \in \Pi_{\phi_v}(\SO_{2m}(\Q_v))$), we have that $\rec_v(\pi_v) \simeq \xi \circ \phi_v$ (resp. $\rec_v(\pi'_v) \simeq \xi ' \circ \phi'_v$) for all places $v$ of $\Q$.

Let $\rho_q$ (resp. $\rho'_q$) be a representation induced from a character of $S$-type (resp. of $O$-type) and order $p$ as in Section \ref{sec:5}, and WD$(\rho_q)$ (resp. WD$(\rho'_q)$) the associated Weil-Deligne representation which gives rise to a local $L$-parameter $\phi_q$ (resp. $\phi'_q$) for $\GL_{2m}(\Q_q)$. Since $\rho_q$ (resp. $\rho'_q$) is irreducible and symplectic (resp. orthogonal) the parameter $\phi_q$ (resp. $\phi'_q$) factors through $\Sp_{2m} (\C) \subseteq \GL_{2m}(\C)$ (resp. $\SO_{2m} (\C) \subseteq \GL_{2m}(\C)$), possibly after conjugation, and defines a discrete $L$-parameter of $\SO_{2m+1}(\Q_q)$ (resp. $\SO_{2m}(\Q_q)$). Then the $L$-packet $\Pi_{\phi_q}(\SO_{2m+1}(\Q_q))$ (resp. $\Pi_{\phi'_q}(\SO_{2m}(\Q_q))$) consists of finitely many discrete series of $\SO_{2m+1}(\Q_q)$ (resp. $\SO_{2m}(\Q_q)$).

\begin{rema} 
The proof of Arthur's results is still conditional on the stabilization of the twisted trace formula and a few expected technical results in harmonic analysis. However, recently Moeglin and Waldspurger have announced that the proof of Arthur's results is now unconditional (see \cite{MW17a} and \cite{MW17b}).
\end{rema}

\begin{theo}\label{chin}
There exist self-dual cuspidal automorphic representations $\pi$ of $\GL_n(\A_\Q)$ with trivial central character such that 
\begin{enumerate}
\item $\pi$ is unramified outside $q$, 
\item $\rec_q(\pi_q) \simeq \emph{WD}(\rho_q)$ $(\mbox{resp. }\rec_q(\pi_q) \simeq \emph{WD}(\rho'_q))$, and
\item $\pi_\infty$ is of symplectic (resp. orthogonal) type and regular algebraic. 
\end{enumerate} 
\end{theo}
\begin{proof}
The proof of this theorem is analogous to the proof of Theorem 3.4 of \cite{ADSW14}, i.e., by applying Theorem \ref{shins} with $S=\{ q \}$ and $\hat{U}_q = \Pi_{\phi_q}(\SO_{2m+1}(\Q_q))$ (resp. $\hat{U}_q = \Pi_{\phi'_q}(\SO_{2m}(\Q_q))$). 
\end{proof}

\begin{coro}\label{chidote}
There exist compatible systems as in Theorem \ref{rt}.
\end{coro}
\begin{proof}
Let $\pi$ be an automorphic representation as in Theorem \ref{chin} and $\mu_{\text{triv}}$ the trivial character of $\A^\times_\Q / \Q^\times$.  
Then $(\pi, \mu_{\text{triv}})$ is a RAESDC automorphic representation of $\GL_n(\A_\Q)$ and the compatible system $\mathcal{R}(\pi) = \{  \rho_{\ell, \iota} (\pi) \}_\ell$ associated to $(\pi,  \mu_{\text{triv}})$ as in Theorem \ref{raesdc} is totally odd polarizable for the compatible system of characters $ \{ \chi_\ell^{1-n} \rho_{\ell, \iota} (\mu_{\text{triv}}) \}_\ell = \{\chi_\ell^{1-n} \}_\ell$.

Note that $\mathcal{R}(\pi)$ is Hodge-Tate regular and for every $\ell \neq q$, $\rho_{\ell, \iota}(\pi)$ is crystalline. 
Let $a \in \Z$ be the smallest Hodge-Tate weight and let $k$ be the greatest difference between any two Hodge-Tate numbers. By Fontaine-Laffaille theory, we have that for every prime $\ell$ such that $\ell > k+2$ and $\ell \neq q$, the representation $\overline{\chi}_\ell^a \otimes \overline{\rho}_{\ell ,\iota}(\pi)$ is regular and the tame inertia weights of this representation are bounded by $k$.  

Finally, taking $p$ and $q$ as in Lemma \ref{prim} with $N=1$ and by part $ii)$ of Theorem \ref{chin}, we have that $\rho_{\ell,\iota}(\pi)$ is maximally induced of $S$-type or $O$-type at $q$ of order $p$ for all prime $\ell$ different from $2$, $p$ and $q$. 
\end{proof}

\begin{rema} 
From the local global compatibility, we have that the Frobenius semisimplification of $\rho_{\ell, \iota} (\pi) \vert _{G_{\Q_p}}$ is isomorphic to $\rec_p(\pi_p) \otimes \vert \; \vert ^{(1-n)/2}$ for all $p \neq \ell$. By the self-duality of $\pi$ and Chevotarev's Density Theorem, we have that $\rho_{\ell, \iota}^{\vee} (\pi) = \rho_{\ell, \iota}(\pi) \vert \; \vert ^{n-1}$  and then $\rho_{\ell, \iota}(\pi)$ acts by either orthogonal or symplectic similitudes on $\overline{\Q}^\times_\ell$ with similitude factor $\vert \; \vert ^{n-1}$. Although it is possible for an irreducible representation to act by both orthogonal and symplectic similitudes, this is not possible if the factor of similitude are the same. 
Since $\rho_q$ (resp. $\rho'_q$) as above is irreducible, then WD$(\rho_q)$ (resp. WD($\rho'_q$)) and WD$(\rho_{\ell, \iota})(\pi) \vert _{G_{\Q_q}}$ are already Frobenius semisimple.  
Thus, $\mbox{WD}(\rho_{\ell, \iota}(\pi) \vert _{G_{\Q_q}}) \simeq \mbox{WD}(\rho_q) \otimes \vert \;  \vert^{(1-n)/2}$ (resp. $\mbox{WD}(\rho_{\ell, \iota}(\pi) \vert _{G_{\Q_q}}) \simeq \mbox{WD}(\rho'_q) \otimes \vert \;  \vert^{(1-n)/2}$), which implies that $\rho_{\ell, \iota}(\pi) \vert _{G_{\Q_q}} \simeq \rho_q \otimes \vert \;  \vert^{(1-n)/2}$ (resp. $\rho_{\ell, \iota}(\pi) \vert _{G_{\Q_q}} \simeq \rho'_q \otimes \vert \;  \vert^{(1-n)/2}$).
Finally, as $\rho_q$ (resp. $\rho_q'$)  is an irreducible symplectic (resp. orthogonal) representation, it follows that $\rho_{\ell, \iota}(\pi)$ is irreducible and symplectic (resp. orthogonal) with similitude factor $\vert \; \vert^{n-1}$. Therefore the image of $\rho_{\ell, \iota}(\pi)$  may be conjugate in $\GSp_n(\overline{\Q}_\ell)$ (resp. $\GO_n(\overline{\Q}_\ell)$) to a subgroup of $\GSp_n(\overline{\Z}_\ell)$ (resp. $\GO_n(\overline{\Z}_\ell)$).
\end{rema}


\section{Study of the images II (non geometric cases)}\label{sec:8}

In this section we will give some refinements of Theorem \ref{rt} for some low-dimensional symplectic and orthogonal groups. More precisely we will prove the following result.

\begin{theo}\label{rib}
Let $\mathcal{R} = \{ \rho_\ell \}_\ell$ be a totally odd polarizable compatible system of Galois representations $\rho_\ell: G_\Q \rightarrow \GL_n(\overline{\Q}_\ell)$ as in Theorem \ref{rt}. Then for almost all primes $\ell$ we have that:
\begin{enumerate}
\item If the image of $\rho_\ell$ is contained in $\GSp_n(\overline{\Q}_{\ell})$ and $6 \leq n \leq 12$, then the image of $\overline{\rho}_\ell^{\proj}$ is equal to $\PSp_n(\F_{\ell^s})$ or $\PGSp_n(\F_{\ell^s})$ for some integer $s > 0$.
\item If the image of $\rho_\ell$ is contained in $\GO_{12}(\overline{\Q}_{\ell})$, then the image of $\overline{\rho}_\ell^{\proj}$ is equal to $\POM^+_{12}(\F_{\ell^s})$, $\PSO^+_{12}(\F_{\ell^s})$, $\PO_{12}^+(\F_{\ell^s})$ or $\PGO^+_{12}(\F_{\ell^s})$ for some  integer $s>0$.
\end{enumerate}
\end{theo}

In order to prove this theorem we need a more precise description of the maximal subgroups which are not of geometric type.

\begin{lemm} 
Let $G$ be a subgroup of $\GSp_n(\F_{\ell^r})$ (resp. $\GO^\pm_n(\F_{\ell^r})$) such that it is not contained in a maximal subgroup lying in cases i)-vi) of Aschbacher's classification, then one of the following holds:
\begin{enumerate}
\item $G$ is contained in a maximal subgroup of class $\mathcal{S}$ of $\GSp_n(\F_{\ell^s})$ (resp. $\GO^\pm(\F_{\ell^s})$) for some integer $s>0$ dividing $r$, or
\item $\emph{P}G$ is conjugate to $\PSp_n(\F_{\ell^s})$ or $\PGSp_n(\F_{\ell^s})$ (resp. to $\POM^\pm_n(\F_{\ell^s})$, $\PSO^\pm_n(\F_{\ell^s})$, $\PO_n^\pm(\F_{\ell^s})$ or $\PGO^\pm(\F_{\ell^s})$) for some integer $s>0$ dividing $r$.
\end{enumerate} 
\end{lemm}
\begin{proof} If $G$ contains $\Sp_n(\F_{\ell^r})$ (resp. $\Omega^\pm_n(\F_{\ell^r})$), then $G$ lies in case ii).
On the other hand, if $G$ does not contain $\Sp_n(\F_{\ell^r})$ (resp. $\Omega_n^\pm(\F_{\ell^r})$), it is contained in a maximal subgroup lying in  cases vii) and viii) of Aschbscher's classification. 

If we assume that $G$ is contained in a maximal subgroup in case viii) of Aschbacher's classification, then $G$ lies in case i). 

Now, if we assume that $G$ is a subgroup of a maximal subgroup in case vii) of Aschbacher's classification, then there exists a minimal integer $s>0$, dividing $r$, such that $G$ is contained in $\GSp_n(\F_{\ell^s})$ (resp. in $\GO^\pm(\F_{\ell^s})$). See $\S 2.2.5$ of \cite{BHR13}. According to Aschbacher's classification we have the following three possibilities for $G$:
\begin{itemize}
\item $G$ is contained in a maximal subgroup of $\GSp_n(\F_{\ell^s})$ (resp. of $\GO^\pm_n(\F_{\ell^s})$) of geometric type.
This possibility can be excluded by applying the same arguments as in Section \ref{sec:6}.
\item $G$ is contained in a maximal subgroup  of class $\mathcal{S}$ of $\GSp_n(\F_{\ell^s})$ (resp. of $\GO^\pm_n(\F_{\ell^s})$). So $G$ lies in case i). 
\item $G$ contains $\Sp_n(\F_{\ell^s})$ (resp. $\Omega^\pm_n(\F_{\ell^s})$), so $G$ lies in ii).
\end{itemize}
\end{proof}

Then to prove Theorem \ref{rib} we just need to show that the image of $\overline{\rho}_\ell$ is not contained in a maximal subgroup of class $\mathcal{S}$ of $\GSp_n(\F_{\ell^s})$ (resp. of $\GO^+_n(\F_{\ell^s})$) for some $s$ dividing $r$.
According to Chapter 4 and 5 of $\cite{BHR13}$, at least in dimension smaller than or equal to 12, the groups of class $\mathcal{S}$ are divided in two classes as follows. We say that a group $G$ of class $\mathcal{S}$ lies in the \emph{class of defining characteristic}, denoted by $\mathcal{S}_2$, if $G^\infty$ is isomorphic to a group of Lie type in characteristic $\ell$, and  $G$ lies in the \emph{class of cross characteristic}, denoted by $\mathcal{S}_1$, otherwise.  For a fixed dimension the set of orders of the cross characteristic groups is bounded above independently of $\ell$. In contrast, the groups in defining characteristic have unbounded order as $\ell$ varies.

Now we are ready to give the proof of Theorem \ref{rib}, which will be given by considering the following two cases:


\subsection{Symplectic case}\label{sec:81}

Throughout this section we will assume that $\ell \geq 7$. Suppose that $G_\ell$ is a subgroup of a maximal subgroup lying in $\mathcal{S}_1$. Then according to Propositions 6.3.17, 6.3.19, 6.3.21 and 6.3.23 of \cite{BHR13}, $\mbox{P}G_\ell$ must be contained in an extension of degree at most $2$ of one of the following groups (see \cite{KL90} and \cite{BHR13} for the notation):
\begin{itemize}
\item $\PSL_2(\F_7)$ (of order $2^4 \cdot 3 \cdot 7$), $\PSL_2(\F_7).2$,
\item $\PSL_2(\F_{11})$ (of order $2^3 \cdot 3 \cdot 5 \cdot 11$), $\PSL_2(\F_{11}).2$,
\item $\PSL_2(\F_{13})$ (of order $2^3 \cdot 3 \cdot 7 \cdot 13$), $\PSL_2(\F_{13}).2$,
\item $\PSL_2(\F_{17})$ (of order $2^4 \cdot 3^2 \cdot 17$),
\item $\PSL_2(\F_{25})$ (of order $2^4 \cdot 3 \cdot 5^2 \cdot 13$),
\item $\PSp_4(\F_5)$ (of order $2^7 \cdot 3^2 \cdot 5^4 \cdot 13$),
\item $\mbox{PSU}_3(\F_3)$ (of order $2^5 \cdot 3^3 \cdot 7$), $\mbox{PSU}_3(\F_3).2$,
\item $\mbox{PSU}_5(\F_2)$ (or order $2^{11} \cdot 3^5 \cdot 5 \cdot 11$), $\mbox{PSU}_5(\F_2).2$,
\item $\mbox{G}_2(\F_4)$ (of order $2^{13} \cdot 3^3 \cdot 5^2 \cdot 7 \cdot 13$), $\mbox{G}_2(\F_4).2$,
\item $J_2$ (of order $2^8 \cdot 3^3 \cdot 5^2 \cdot 7$),
\item $A_5$, $S_5$, $A_6$ or $A_6.2_2$.
\end{itemize}
But as we have chosen $p > 17$, we have that the image of $\overline{\rho}^{\text{proj}}$ cannot be contained in these subgroups.

On the other hand, let $\mathcal{G}$ be an algebraic group over $\Z$ admitting an absolutely irreducible symplectic representation of dimension $n$. Then we can consider the corresponding map $\sigma:\mathcal{G} \rightarrow \GSp_{n,\Z}$ and the subgroup $\sigma(\mathcal{G}(\F_{\ell^r}))$ of $\GSp_n(\F_{\ell^r})$.
There is a general philosophy which states that, for $\ell$ sufficiently large, all the maximal subgroups in class $\mathcal{S}_2$ should arise from this construction for suitable $\mathcal{G}$ and $\sigma$ (see Section 1 of \cite{Lar95} and \cite{Se86}). 

For example, if $\mathcal{G}=\SL_2$ and $n$ is an even positive integer greater than $2$, this group admits an absolutely irreducible symplectic representation of dimension $n$ given by the $(n-1)$-th symmetric power of $\SL_2$. Then it gives rise to an embedding $\SL_2 \hookrightarrow \Sp_n$. This representation extends to a representation $\GL_2 \rightarrow \GSp_n$ and the $\F_{\ell^r}$-points of the image of this representation gives rise to an element of $\mathcal{S}_2$ (see Proposition 5.3.6.i of \cite{BHR13}).
In fact, according to Tables $8.29$, $8.49$, $8.65$ and $8.81$ of loc. cit., this is the only kind of subgroups lying in the class of defining characteristic if $6 \leq n \leq 12$.

In order to deal with this case we will use Dickson's well known classification of maximal subgroups of $\PGL_2(\F_{\ell^r})$ which states that they can be either isomorphic to a group of upper triangular matrices, a dihedral group $D_{2d}$ (for some integer $d$ not divisible by $\ell$), $\PSL_2(\F_{\ell^s})$, $\PGL_2(\F_{\ell^s})$ (for some integer $s>0$ dividing $r$), $A_4$, $S_4$ or $A_5$.

Let $\mbox{P}G_q$ be the projective image of $\Ind^{G_{\Q_q}}_{G_{\Q_{q^n}}} (\overline{\chi}_q)$ which is contained in $\mbox{P}G_\ell$. If $\mbox{P}G_q$ is contained in a group of upper triangular matrices, it is contained in fact in the subset of diagonal matrices because $\ell$ and $2p$ are coprime. But we know that $\mbox{P}G_q$ is non-abelian, then it cannot be contained  in a group of upper triangular matrices. Moreover, $\mbox{P}G_q$ cannot be contained in $A_4$, $S_4$ or $A_5$ because we have chosen $p$ greater than $7$.

Now assume that $\mbox{P}G_q$ is contained in a dihedral group.  As any subgroup of a dihedral group is either cyclic or dihedral and as $\mbox{P}G_q$ is non-abelian, we can assume that it is in fact a dihedral group of order $np$. This implies that $\mbox{P}G_q$ contains an element of order $mp$ (with $m \in \N$ such that $n=2m$), but we know that the elements of $\mbox{P}G_q$ have order at most $p$. Then $\mbox{P}G_q$ cannot be contained in a dihedral group.
Therefore $\mbox{P}G_q$ should be isomorphic to $\PSL_2(\F_{\ell^s})$ or $\PGL_2(\F_{\ell^s})$ for some integer $s>0$. As we are assuming $\ell \geq 7$, $\PSL_2(\F_{\ell^s})$ is an index 2 simple subgroup of $\PGL_2( \F_{\ell^s})$. But $\mbox{P}G_q$ contains a normal subgroup of order $p$, thus of index greater than $2$ (because we are assuming $n \geqslant 6$). Therefore, we have shown that the image of $\overline{\rho}_\ell$ cannot be contained in a maximal subgroup of class $\mathcal{S}_2$.  Then the first part of Theorem \ref{rib} is proved. 


\subsection{Orthogonal case}\label{sec:82}

According to Remark \ref{remi} and the construction in Section \ref{sec:7}, the first case where we can apply our results to orthogonal groups is when $n$ is equal to $12$.  
In this case, as $n \equiv 0 \mod 4$, it follows from Section $1$ that the image $\mbox{P}G_\ell$ of $\overline{\rho}_\ell^{\text{proj}}$ lies in $\PGO^+(\F_{\ell^r})$. 

From Table 8.83 of \cite{BHR13} we have that $\mathcal{S}_2$ is empty. Then by Proposition 6.3.23 of loc. cit. $\mbox{P}G_\ell$ must be contained in an extension of degree $2^a$ (with $a$ at most $3$) of one of the following groups (see \cite{KL90} and \cite{BHR13} for the notation):
\begin{itemize}
\item $\PSL_2(\F_{11})$ (of order $2^2 \cdot 3 \cdot 5 \cdot 11$),
\item $\PSL_2(\F_{13})$ (of order $2^2 \cdot 3 \cdot 7 \cdot  13$),
\item $\PSL_3(\F_3)$ (of order $2^4 \cdot 3^3 \cdot 13$), $\PSL_3(\F_3).2$,
\item $M_{12}$ (of order $2^6 \cdot 3^3 \cdot 5 \cdot 11$), $M_{12}.2$ or $A_{13}$.
\end{itemize}
Then from the choice of $p$, we can conclude that the image of $\overline{\rho}^{\text{proj}}$ cannot be contained in these subgroups. Therefore the second part of Theorem \ref{rib} is proved.

Finally, from Corollary \ref{chidote} there exist compatible systems satisfying the conditions of Theorem  \ref{rib} then we have the following result.

\begin{coro}\label{galo}
At least one of the following orthogonal groups: $\POM^+_{12}(\F_{\ell^s})$, $\PSO^+_{12}(\F_{\ell^s})$, $\PO_{12}^+(\F_{\ell^s})$ and $\PGO^+_{12}(\F_{\ell^s})$ are Galois groups of $\Q$ for infinitely many primes $\ell$ and infinitely many integers $s>0$.
\end{coro}

\begin{rema}
As we saw through this section, our main tool to prove Theorem \ref{rib} was the classification of the maximal subgroups of class $\mathcal{S}$ of dimension at most 12.
Unfortunately, to the best of our knowledge, it is a feature of the subgroups in class $\mathcal{S}$ that they are not susceptible to a uniform description across all dimensions. So this is one the main obstacles to extending our results to higher dimensions. However, inspired by a recent work of Lombardo \cite[Section 9]{Lo17}, we believe that it is possible to extend Theorem \ref{rib} by applying some results from representation theory of algebraic and finite groups, or at least to be able to say more about the possible images of $\overline{\rho}_\ell$ in Theorem \ref{rt}.
\end{rema}


\end{document}